\documentclass[a4paper]{article}
\usepackage{amsthm}

\usepackage{a4wide}

\usepackage{amssymb}
\usepackage{hyperref}

\usepackage{graphicx}

\newtheorem{thm}{\Large T\normalsize heorem}[section]

\newtheorem{lemma}[thm]{\Large L\normalsize emma}
\newtheorem{prop}[thm]{\Large P\normalsize roposition}
\newtheorem{defi}[thm]{\Large D\normalsize efinition}

\newcommand{\Hom}{\mathop{\to}}
\newtheorem{example}{Example}[section]

\newtheorem{corollary}[thm]{Corollary}

\begin{document}
\title{Universality of intervals of line graph order}
\author{Ji\v r\'\i{} Fiala\thanks{Supported by M\v{S}MT \v{C}R grant LH12095 and GA\v{C}R grant P202/12/G061, Department of Applied Mathematics, Charles University, 
Prague, Czech Republic, \texttt{fiala@kam.mff.cuni.cz}}, 
Jan Hubi\v{c}ka\thanks{This research was partially done while the second and third author took part in Trimester Universality and Homogeneity at Hausdorff Institute (Bonn) in fall 2013. Second author also acknowledge support of grant ERC-CZ LL-1201 of the Czech Ministry of Education and CE-ITI P202/12/G061 of GA\v CR. PIMS postdoctoral fellow, University of Calgary, Calgary, Canada, \texttt{hubicka@kam.mff.cuni.cz}}, Yangjing Long\thanks{Third author acknowledges the doctoral scholarship from the International Max-Planck Research School at the Max-Planck-Institute for Mathematics in the Natural Sciences, Leipzig. Max Planck Institute for Mathematics in the Sciences, Inselstrasse 22, D-04103 Leipzig, Germany, \texttt{ylong@mis.mpg.de}}}
\date{January 2014}
\maketitle

\begin{abstract}
We prove that for every $d\geq 3$ the homomorphism order of the class of line graphs of finite graphs with maximal degree $d$ is universal. This means that every finite or countably infinite partially ordered set may be represented by line graphs of graphs with maximal degree $d$ ordered by the existence of a homomorphism.
\end{abstract}

\section{Introduction}

An (undirected) {\em graph} $G$ is a pair $G=(V_G,E_G)$ such that $E_G$ is a set of
2-element subsets of $V_G$. We denote by $V_G$ {\em the set of vertices of $G$}
and by $E_G$ {\em the set of edges of $G$}. We consider only finite graphs.

For given graphs $G$ and $H$ a {\em homomorphism} $f:G\to H$ is a mapping
$f:V_G\to V_H$ such that $\{u,v\}\in E_G$ implies $\{f(u),f(v)\}\in E_H$.  We
denote the existence of a homomorphism $f:G\to H$ by $G\Hom H$. This allows us to
consider the existence of a homomorphism, $\Hom$, to be a (binary) relation on
the class of finite graphs.

The relation $\to$ is reflexive (identity is a homomorphism) and transitive
(composition of two homomorphisms is still a homomorphism). Thus the existence
of a homomorphism induces a quasi-order on the class of all finite graphs.
This quasi-order can be transformed, in a standard way, to a partial order by
considering only the isomorphism types of vertex-inclusion minimal elements of
each equivalency class of $\to$ (the graph cores). The resulting partial order
is known as the {\em homomorphism order} of graphs. We denote $G\leq H$ whenever $G\to H$.

This partial order generalizes graph coloring and its rich structure is a
fruitful area of research, see the Partial Order of Graphs and Homomorphisms chapter in the monograph of Hell and Ne\v{s}et\v{r}il~\cite{Hell2004}.
The richness of homomorphism order is seen from the perhaps surprising fact that
every countable (finite or infinite) partial order can be found as one of its suborders. This property of
partial order is known as {\em universality}.  The existence of such countable
partial orders may seem counter-intuitive: there are uncountably many
different partial orders and they are all ``packed'' into a single countable
structure. The existence of such a partial order is given by the classical Fra\"{\i}ss\' e Theorem. There are however few known explicit representations of such partial orders, see \cite{Hedrlin1969,Hubicka2005a,Hubicka2011}. The universality of the homomorphism order of graphs was first proved by Hedrl\'\i{}n and Pultr
in a categorical setting as a culmination of several papers (see~\cite{pultr1980combinatorial} for a complete proof).

It is interesting to observe that almost all naturally defined partial orders
of graphs fail to be universal for simple reasons --- by the absence of
infinite increasing chains, decreasing chains or anti-chains.  Several variants
to graph homomorphisms are considered in our sequel paper~\cite{Fiala} and only locally
constrained homomorphisms are shown to produce universal orders. It is also a
deep result of Robertson and Seymour that the graph minor order is a well
quasi-order and thus not universal~\cite{Robertson2004}.

Although universality is a quite rare property of orders on graphs, it is a very
robust property of the homomorphism order. Hubi\v{c}ka and Ne\v{s}et\v{r}il showed  
that even the homomorphism order of oriented paths is universal~\cite{Hubicka2005}.
This
result can be used to easily show the universality of other classes of graphs,
such as planar graphs or series-parallel graphs~\cite{Hubicka2004}. This shows that complex graphs or complex computational problems (homomorphism testing is polynomial on oriented paths) are not needed to build universal partial orders, which is something that was not anticipated by Ne\v{s}et\v{r}il and Zhu a decade earlier~\cite{JaroslavNesetril1996}.

Recently, D.~E.~Roberson proposed a systematic study of the homomorphism order of the class
of line graphs~\cite{Roberson}.
Here, a \emph{line graph} of an undirected graph $G$, denoted by $L(G)$, is a graph
$H=(V_H,E_H)$ such that $V_H=E_G$ and two distinct vertices of $L(G)$ are adjacent if
and only if their corresponding edges share a common endpoint in $G$. Because edges of $G$ play the role of vertices of $L(G)$, we will refer vertices of line graphs as {\em nodes}.

The classical Vizing theorem gives an insight into the structure of the homomorphism
order of line graphs in terms of \emph{chromatic index} $\chi'(G)$
that is the chromatic number of $L(G)$, i.e. the minimum number of colors needed to color edges of a graph $G$ such that edges with a common vertex receive different colors:

\begin{thm}[Vizing \cite{Vizing1964}]
For any graph $G$ of maximum degree $d$ it holds that $\chi'(G) \le d+1$.
\end{thm}

Since the line graph of a graph with a vertex of degree $d$ contains a $d$-clique, the Vizing theorem splits graphs into two classes. {\em Vizing class 1} contains the graphs whose chromatic index is the same as the maximal degree of a vertex, while {\em Vizing class 2} contains the remaining graphs.

The approach taken by Roberson~\cite{Roberson} divides the class of line graphs into
intervals.  By $[K_{n},K_{n+1})_\mathcal{L}$ we denote the class of all line
graphs $L(G)$ such that $K_{n}\leq L(G)<K_{n+1}$. The line graphs in each interval have a particularly simple characterization:

\begin{corollary}\label{cor:Vizing}
The intervals $[K_{d},K_{d+1})_\mathcal{L}$ consist of line graphs of graphs whose maximum degree is $d$.  
\end{corollary}

\begin{proof}
The existence of a $(d+1)$-edge coloring is equivalent with $L(G)\le K_{d+1}$. Note that for the Vizing class 1 we indeed have $L(G)\le K_d \le K_{d+1}$.

As $G$ contains a vertex of degree $d$, we have $K_d\le L(G)$, indeed $K_d\subseteq L(G)$. On the other hand, a clique on $d+1\ge 4$ vertices can be formed only from $d+1$ edges sharing a common vertex, hence $K_{d+1} \not\le L(G)$. The same argument used for $K_d\le L(G)$ implies that $G$ contains a vertex of degree $d$.
\end{proof}

The line graphs can be considered as almost perfect graphs (a {\em perfect graph} is a
graph in which the chromatic number of every induced subgraph is equal to the size
of the largest clique of that subgraph). The homomorphism order of the class of
perfect graphs is a trivial chain, since the core of every perfect graph is a
clique.  The almost-perfectness of the class of line graphs suggests that the
homomorphism order of this class may be more constrained in its structure
than the homomorphism order of graphs in general, and indeed many of the results about properties of the homomorphism order can not be easily restricted to the line graphs.

Roberson, in~\cite{Roberson}, showed that the homomorphism order of line graphs
contains many gaps. This is a first important difference from the structure of
the homomorphism order of graphs which was shown (up to one exception) to be
dense by Welzl~\cite{Welzl1982}.  Roberson also asked whether every interval
$[K_d,K_{d+1})_L$, $d\geq 3$ contains infinitely many incomparable elements.
The answer is trivially negative for graphs with maximal degree 1 and 2.  We give
an affirmative answer to this problem.  Indeed, we show:

\begin{thm}
\label{thm:main}
The homomorphism order of line graphs is universal 
on every interval $[K_d,K_{d+1})_{\mathcal L}$ for $d\geq 3$.
\end{thm}

This further develops the results on the universality of the homomorphism order of special classes of graphs (see e.g.~\cite{Hubicka2004,Hubicka2005,Nesetril2006,Nesetril2007}), and on universal partially ordered structures in general (see e.g.~\cite{Johnston1956,Hedrlin1969,Hubicka2005a,Lehtonen2008,Lehtonen2010,Hubicka2011,Kwuida2011}).

As a special case, the universality of interval $[K_3,K_4)_L$ follows from the
construction given by \v{S}\'amal~\cite{Samal2013}. This is not an obvious
observation --- one has to carefully check that for the graphs constructed in~\cite{Samal2013}
the existence of circulation coincide with the existence of a homomorphism of
line graphs. Our proof uses a new approach based on a new divisibility argument which we have introduced for a similar occasion~\cite{Fiala}. This argument leads to 
a simpler construction without the need of 
complex gadgets (Blanu\v sa snarks) used by \v{S}\'amal~\cite{Samal2013}.

The paper is organized as follows.  In Section~\ref{sec:universal} we briefly review
our construction from \cite{Fiala} about the universality of the divisibility order.  In
Section~\ref{sec:dragon} we prove basic properties of a ``dragon'' graph that
is a particularly simple example of Vizing class 2 graph that is also a graph
core. Section~\ref{sec:indicator} briefly reviews the indicator construction
and Section~\ref{sec:final} contains the proof of our main result. In
Section~\ref{sec:concluding} we discuss an extension of our construction to
$d$-regular graphs and some additional observations on the homomorphism order
of line graphs.

\section{A particular universal partial order}
\label{sec:universal}

Let $(P,\leq_P)$ be a partial order where $P$ consists of all finite set of
integers and for $A,B\in P$ we put $A\leq_P B$ if and only if for every $a\in
A$ there is $b\in B$ such that $b$ divides $a$.  We make use of the following:

\begin{thm}[\cite{Fiala}]
\label{thm:universal}
The order $(P,\leq_P)$ is a universal partial order.
\end{thm}

To make the paper self-contained we give a short proof of this result.
See also \cite{Hubicka2004,Hubicka2011} for related constructions of universal
partial orders.

We say that a countable partial order is {\em past-finite} if every down-set is
finite. Similarly a countable partial order is {\em future-finite} if every
up-set is finite. Again, we say that a countable partial order is {\em
past-finite-universal}, if it contains every
past-finite partial order as a suborder.
The {\em future-finite-universal} orders are defined analogously.

Let $P_f(A)$ denote the set of all finite subsets of $A$. The following lemma
extends a well known fact about representing finite partial orders by sets
ordered by the subset relation.
\begin{lemma}
\label{lem:pastfiniteuniv}
Let $A$ be countably infinite set.
Then $(P_f(A),\subseteq)$ is a past-finite-universal partial order.
\end{lemma}

\begin{proof}
Consider an arbitrary past-finite set $(Q,\leq_Q)$. Without loss of generality we assume that $Q\subseteq A$.  Assign to every $x\in Q$ a set $E(x) = \{y\in Q; y\leq x\}$. It is easy to verify that $E$ is an embedding 
$(Q,\leq_Q)\to(P_f(A),\subseteq)$.
\end{proof}

By the {\em divisibility partial order}, denoted by $(\mathbb{Z},\leq_d)$, we mean the partial order where vertices are natural numbers and $n$ is smaller than or equal to $m$ if $n$ is divisible by $m$.

\begin{lemma}
\label{lem:futurefiniteuniv}
The divisibility partial order $(\mathbb{Z},\leq_d)$ is future-finite-universal.
\end{lemma}

\begin{proof}
Denote by $\mathbb P$ the set of all prime numbers.  Apply Lemma \ref{lem:pastfiniteuniv} for $A=\mathbb P$.
Observe that $B\in P_f(\mathbb P)$ is a subset of $C\in P_f(\mathbb P)$ if and only if $\prod_{p\in B} p$ divides $\prod_{p\in C} p$.
\end{proof}

\begin{proof}[Proof of Theorem \ref{thm:universal}]
Let be given any partial order $(Q,\leq_Q)$. Without loss of generality we may assume that
$Q\subseteq \mathbb{P}$. This way we enforce the linear order $\leq$ on elements of
$Q$.  Think of $\leq$ as a specification of the time of creation of the elements
of $Q$.

We define two new orders on elements of $Q$: $\leq_f$, the \emph{forwarding order} and $\leq_b$, the \emph{backwarding order}:
\begin{enumerate}
\item We put $x\leq_f y$ if and only if $x\leq_Q y$ and $x \leq y$.
\item We put $x\leq_b y$ if and only if $x\leq_Q y$ and $x \geq y$.
\end{enumerate}

Thus we decompose the partial order $(Q,\leq_Q)$ into $(Q,\leq_f)$ and $(Q,\leq_b)$.
For every vertex $x\in Q$ both sets $\{y\mid y\leq_f x\}$ and $\{y\mid x\leq_b y\}$ are finite.
It follows that $(Q,\leq_f)$ is past-finite and $(Q,\leq_b)$ is future-finite.

Since $(\mathbb{Z},\leq_d)$ is future-finite-universal (Lemma~\ref{lem:futurefiniteuniv}), there is an embedding $E: (Q,\mathop{\leq_b}) \to (\mathbb{Z},\leq_d)$. We put for every $x\in Q$: $$U(x)=\{E(y)\mid y\leq_f x\}.$$
We show that $U$ is an embedding $U:(Q,\leq_Q) \to (P,\leq_P)$.

First we show that $U(x)\leq_P U(y)$ imply $x\leq_Q y$.
From the definition of $\leq_P$ we know that there is
$w\in Q$, $E(w)\in U(y)$, such that $E(x)\leq_d E(w)$. By the definition of $U$, $E(w)\in U(y)$ if and only if $w\leq_f y$. By the definition of $E$, $E(x)\leq_d E(w)$ if and only if $x\leq_b w$.
It follows that $x\leq_b w\leq_f y$ and thus also $x\leq_Q w\leq_Q y$ and consequently $x\leq_Q y$.

To show that $x\leq_Q y$ imply $U(x)\leq_P U(y)$ we consider two cases.
\begin{enumerate}
\item When $x\leq y$ then $U(x)\subseteq U(y)$ and thus also $U(x)\leq_P U(y)$.
\item Assume $x>y$ and take any $w\in Q$ such that $E(w)\in U(x)$.  From the definition of $U(x)$ we have
$w\leq_f x$. Since $x\leq_Q y$ we have $w\leq_Q y$.
If $w\leq y$, then $w\leq_f y$ and we have $E(w)\in U(y)$. In the other case if $w>y$ then $w\leq_b y$ and thus $E(w)\leq_d E(y)$. Because the choice of $w$ is arbitrary, it follows that $U(x)\leq_P U(y)$.
\end{enumerate}%
\end{proof}

\begin{example}
Consider partial order $(Q,\leq_Q)$ as specified by Figure \ref{fig:poset}. 
\begin{figure}
\centerline{\includegraphics{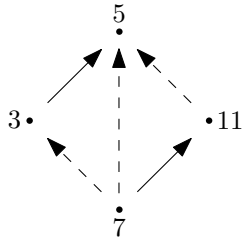}}
\caption{Partial order $(Q,\leq_Q)$. Dashed lines represent backwarding edges.}
\label{fig:poset}
\end{figure}
The following is the representation of $(Q,\leq_Q)$ in $(P,\leq_P)$ given by the proof of Theorem~\ref{thm:universal}.
$$E(3)=3, E(5)=5, E(7)=3\times 5\times 7, E(11)=5\times 11\hbox{, then}$$
$$U(3)=\{3\},U(5)=\{5,3\},$$ $$U(7)=\{3\times 5\times 7\},U(11)=\{3\times 5\times 7,5\times 11\}.$$
\end{example}

\section{Dragon graphs}
\label{sec:dragon}
We use a simple gadget called $d$-dragon which is also used in several constructions developed by Roberson~\cite{Roberson}. In our constructions, the parameter $d$ specifies the maximal degree of a vertex:

\begin{defi}
For $d\geq 3$, the \emph{$d$-dragon}\footnote{The name is derived from a visual similarity of this graph to a kite that in Czech language is called ``dragon''.}, denoted by $D_d$, is the graph created from $K_{d+1}$ by replacing one of its edges by a path on 3 vertices.
\end{defi}

The 3-dragon is depicted in Figure \ref{fig:dragon}.

\begin{figure}[t]
	\centering
	\includegraphics{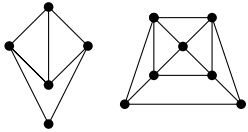}
	\caption{The 3-dragon $D_3$ and its line graph $L(D_3)$.}
	\label{fig:dragon}
\end{figure}

We proceed by a simple lemma about edge-colorings of dragons.

\begin{lemma}\label{lem:chrInd_Dragon}
For all $d\geq 3$ it holds that $D_d$ is a Vizing class 2 graph, i.e. its chromatic index is $d+1$.
\end{lemma}

\begin{proof}
By Vizing theorem, $L(D_d)$ is $(d+1)$-colorable.
We prove that $L(D_d)$ is not $d$-colorable. The number of edges of $D_d$ is $\frac{d(d+1)}2+1=\frac{d^2+d+2}2$, the number of vertices is $d+2$. We use the fact that every $k$-edge-coloring yields a decomposition of the graph into $k$ disjoint matchings. We consider two cases:
\begin{enumerate}
\item If $d$ is odd, then the maximum size of a matching in $D_d$ is $\frac{d+1}2$, so the
partition contains at least $d+1$ matchings. Thus the chromatic number of $L(D)$ is $d+1$.
\item
If $d$ is even, then the maximum size of a matching in $D_d$ is $\frac{d+2}{2}$. However note that $D_d$ 
has a vertex of degree 2, so any partition of edge set of $D_d$ into disjoint matchings contains at most 2 matchings of maximum cardinality. The others matchings have the size at most $\frac{d}2$.
It follows that $d$ matchings can cover at most $2(\frac{d+2}2)+(d-2)\frac{d}2=\frac{4+d^2}2$ edges. For $d\geq 4$ we have $\frac{4+d^2}2< \frac{d^2+d+2}2$ and thus the partition contains at least $d+1$ matchings, i.e. color classes.
\end{enumerate}
\end{proof}

A graph $G$ is a {\em core} if there is no homomorphism from $G$ to its
proper subgraphs. In our construction we will use the fact that the line graphs of $d$-dragons are cores
which we show in the following lemma.
\begin{lemma} \label{lem:core}
For every $d\geq 3$, the graph $L(D_d)$ is a core.
\end{lemma}

\begin{proof}
For $d=3$ observe that $L(D_3)$, depicted in Figure~\ref{fig:dragon}, is not 3-colorable, 
while each of its induced subgraphs is. Hence the statement holds for $d=3$. 

For $d\ge 4$ denote the vertices of $D_d$ by $1,2,\dots, d,d+1,d+2$, where vertices 
$1,2,\dots,d+1$ have degree $d$ and the vertex $d+2$ is adjacent to vertices 1 and $d+1$ and has degree 2. 
The vertices of degree $d$ correspond to $d$-cliques in $L(D_d)$. 
Note that every node of $L(D_d)$ belongs to some $d$-clique, since every edge of $D_d$ is incident to a vertex of degree $d$. 
Each pair of those $d$-cliques share at most one node that corresponds to
the edge connecting the original pair of vertices. 
Note that the shared node is unique for each such pair.
Observe also that there are no other $d$-cliques in $L(D_d)$.
This follows from the fact that the only way to create a $d$-clique in a line graph
is by a vertex of degree at least $d$. See Figure~\ref{fig:4-dragon}.
\begin{figure}[t]
	\centering
	\includegraphics{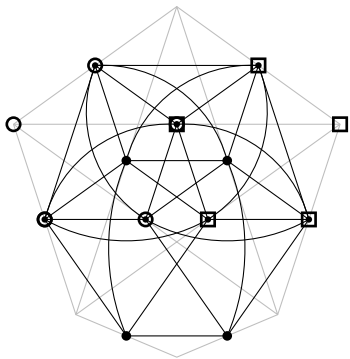}
	\caption{Line-graph of 4-dragon with cliques corresponding to the neighborhoods of two vertices distinguished.}
	\label{fig:4-dragon}
\end{figure}

Consider a homomorphism $f:L(D_d)\to L(D_d)$. Every homomorphism must map a $d$-clique
to a $d$-clique, and thus it defines a vertex mapping $f':\{1,2,\ldots ,d+1\}\to \{1,2,\ldots ,d+1\}$ in $D_d$.

\begin{figure}[t]
	\centering
	\includegraphics{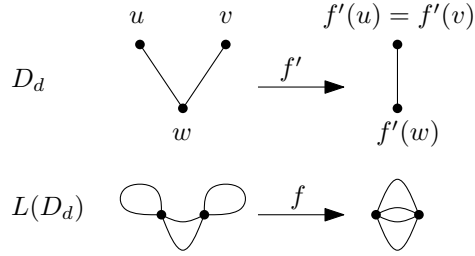}
	\caption{Mapping two $d$-cliques to the same target.}
	\label{fig:d-clique}
\end{figure}

Assume that there are distinct $u,v\in \{1,\ldots,d+1\}$ such that $f'(u)=f'(v)$, see Figure~\ref{fig:d-clique}.
Take any $w \in \{2,\ldots, d\}\setminus\{u,v\}\ne\emptyset$. Because the node shared 
by the cliques corresponding to $u$ and $w$ is unique, it is different from the node shared by the cliques corresponding to $v$ and $w$. Consequently,
the cliques corresponding to $f'(u)=f'(v)$ and $f'(w)$ share at least two nodes. 
Since distinct $d$-cliques of $L(D_d)$ may share at most one node, it follows
that $f'(w)=f'(u)$. Hence $f'$ is either a bijection 
or a constant function on $\{2,\ldots, d\}$ and thus also on $\{1,\ldots, d+1\}$. 
On the other hand, $f'$ can not be a constant function by Lemma~\ref{lem:chrInd_Dragon},
as otherwise such mapping would yield an edge coloring of $D_d$ by $d$ colors.

Since $f'$ is a bijection on vertices $\{1,2,\dots,d+1\}$ of $D_d$, the mapping $f'$ must be a bijection on the edges between these vertices. The only way to get a homomorphism $f$ of the whole $L(D_d)$ is to extend the mapping bijectively also on edges $\{1,d+2\}$ and $\{d+1,d+2\}$.
By this argument we have proved that $f$ is an isomorphism.
\end{proof}

\section{Indicator construction}
\label{sec:indicator}

We briefly describe the indicator technique, often called the ``arrow construction''~\cite{Hell2004}.
Informally, this construction means replacing every edge of a given graph $G$ by a copy of graph $I$
(an indicator) with two distinguished vertices identified with the endpoints of the edge.
Figure~\ref{fig:circle} (left) shows the result of indicator construction on the graph
in Figure~\ref{fig:sunlet} with indicator shown in Figure~\ref{fig:gadget} (left). We give a precise definition of this standard notion:

\begin{figure}[t]
	\centering
	\includegraphics{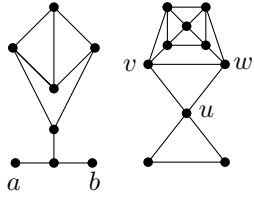}
	\caption{Indicator $I_3(a,b)$ and its line graph.}
	\label{fig:gadget}
\end{figure}

An {\em indicator} is any graph $I=(V_I,E_I)$ with two distinguished vertices $a$,$b$
having the property that there exists an automorphism of $I$ that interchanges $a$ and $b$.
(The last condition is needed in our context of undirected graphs.)

Given a graph $G=(V_G,E_G)$, we denote by $G*I(a,b)$ the graph $H=(V_H,E_H)$,
where each edge is replaced by an extra copy of $I(a,b)$, where the vertices $a$ and $b$
are identified with the original vertices.

Formally, to obtain $V_H$ we first take the Cartesian product 
$\vec{E}_G\times V_I$. Here $\vec{E}_G$ represent arbitrary
orientation of edges $E_G$ (that is, every unordered pair $\{x,y\}$ in $E_G$ corresponds uniquely to an ordered pair $(x,y)$ in $\vec{E}_G$ and vice versa).
Next factorize this set by the equivalence relation $\sim$ consisting of the following pairs:
$$((x,y),a)\sim((x,y'),a),$$
$$((x,y),b)\sim((x',y),b),$$
$$((x,y),b)\sim((y,z),a).$$

In other words, the vertices of $H$ are equivalence classes of the equivalence $\sim$.
For a pair $(e,x)\in E\times V_I$, the symbol $[e,x]$ denotes its equivalence class.

Vertices $[e_1,x_1]$ and $[e_2,x_2]$ are adjacent in $H$ if and only if there exists $(e'_1,x'_1)\in [e_1,x_1]$ and $(e'_2,x'_2)\in [e_2,x_2]$ such that $e'_1=e'_2$ and $\{x'_1,x'_2\}\in E_I$.

\section{Final construction}
\label{sec:final}

It is a standard technique to use an indicator construction to represent a class
of graphs which is known to be universal (such as oriented paths)
within another class of graphs (such as planar graphs) by using an appropriate
rigid indicator, see e.g. work of Hubi\v{c}ka and Ne\v{s}et\v{r}il~\cite{Hubicka2004}. It is then possible to show that the structure induced
by the homomorphism order is preserved by the embedding via the indicator construction.

While our construction also uses an indicator, the application is not so direct.
It is generally impossible to have an indicator that would turn a graph into
a line graph. We use the indicator to make graphs more rigid with respect
to homomorphisms of their line graphs and model the divisibility partial order
directly.

Our basic building blocks are the following:

\begin{defi}
The \emph{$n$-sunlet graph}, denoted by $S_n$, is the graph on $2n$ vertices obtained by attaching $n$ pendant edges to a cycle $C_n$, see Figure~\ref{fig:sunlet}.
\end{defi}
\begin{figure}[t]
	\centering
	\includegraphics{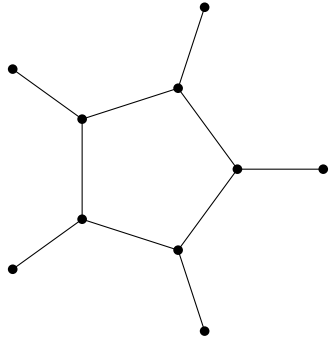}
	\caption{The 5-sunlet $S_5$.}
	\label{fig:sunlet}
\end{figure}

\begin{defi} For $d\geq 3$ the \emph{indicator} $I_d(a,b)$ is the graph created from the disjoint union of the dragon $D_d$ and a path on vertices $a,c,b$, where the middle vertex $c$ is connected by an edge to the vertex of degree 2 in $D_d$, see Figure~\ref{fig:gadget}.
\end{defi}

The desired class of graphs to show universality of interval $[K_{d},K_{d+1})_\mathcal{L}$, $d\ge 3$ consists of graphs $S_n*I_d(a,b)$ for 
$n\ge 3$. We abbreviate $S_n*I_d(a,b)$ by the symbol $G_{n,d}$.
An example, the graph $G_{5,3}$, is shown in Figure~\ref{fig:circle}. By squares are
indicated vertices of degree three of the original sunlet graph. The three adjacent edges 
are in the line graph drawn as the triplets joined by the dashed triangles.

\begin{figure}[t]
	\centering
	\includegraphics{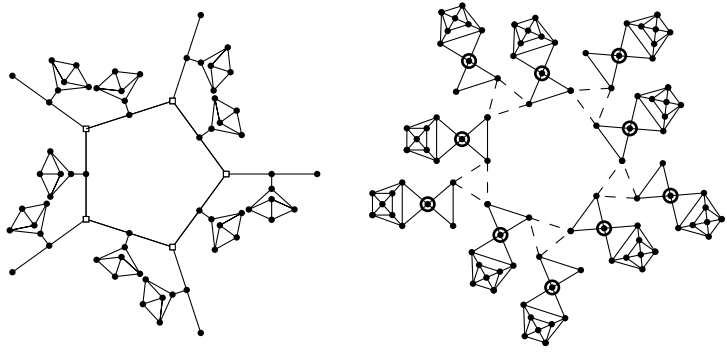}
	\caption{The graph $G_{5,3}=S_5*I_3(a,b)$ and its line graph.}
	\label{fig:circle}
\end{figure}

By Corollary~\ref{cor:Vizing} the graph $L(G_{n,d})$ is in the interval $[K_d,K_{d+1})_{\mathcal L}$ for every $n\geq 3$ and $d\geq 3$.

It remains to show the following property of the graphs $G_{n,d}$.
\begin{prop}
\label{prop:homo}
For every $d\geq 3$, $n\geq 3$, $n'\geq 3$ there is a homomorphism from $L(G_{n,d})$ to $L(G_{n',d})$ if and only if $n$ is divisible by $n'$.
\end{prop}

In one direction the proposition is trivial.  If $n$ is divisible by $n'$ then the 
homomorphism is given by a homomorphism from $S_n$ to $S_{n'}$ that cyclically wraps 
the bigger cycle around the smaller cycle.
We call this homomorphism \emph{cyclic}.

The other implication is a consequence of the following two lemmas.

The nodes of $L(G_{n,d})$ corresponding to the edges connecting the 
dragons with the vertices $c$ are called {\em special}.
In Figure~\ref{fig:circle} they are highlighted by circles.

\begin{lemma} 
\label{lem:special}
For $d\geq 3$, $n\geq 3$, $n'\geq 3$ every homomorphism $f:L(G_{n,d})\to L(G_{n',d})$ must map special vertices to special vertices.
\end{lemma}

\begin{proof}
When $d\ge 4$, then the line graph of any dragon $L(D_d)$ induced in $L(G_{n,d})$ contains a $d$-clique, which can only be mapped by $f$ onto a $d$-clique of $L(G_{n',d})$ --- indeed every vertex of $L(D_d)$ is incident with a $d$-clique. As the only $d$-cliques in both graphs are inside the line graphs of dragons, and all dragons are cores by Lemma~\ref{lem:core}, every dragon $L(D_d)$ maps by $f$ onto an isomorphic dragon in the target $L(G_{n',d})$.

For $d=3$ observe first that the only non 3-colorable subgraphs 
on 7 vertices of $L(G_{n',d})$ are the line graphs of dragons $D_d$. (This might be verified by a straightforward case analysis that we omit here.)
If we would assume for the contrary that $f(L(D_d))$ is not isomorphic to a $L(D_d)$ for some 
dragon $D_d$ in $L(G_{n,d})$, then the 3-coloring of $f(L(D_d))$ would yield a 3-coloring of 
$L(D_d)$, a contradiction with the property that $D_3$ is a core.

Because line graphs of dragons are cores, for any special node $u$ it holds that its two neighbors $v$ and $w$ in the associated line graph of a dragon $L(D_d)$ in $L(G_{n,d})$ must be mapped into some line graph of a dragon $L(D_d)$ from $L(G_{n',d})$. Moreover the images of $v$ and $w$ must be the two neighbors of the attached special node. (See Figure~\ref{fig:gadget}.)

Since $u,v$ and $w$ form a triangle, the only way to complete a triangle containing 
$f(v)$, $f(w)$ is to map $u$ to the adjacent special node, as such triangle cannot be completed inside the dragon.
\end{proof}

A triangle in $L(G_{n,d})$ is called a {\em connecting triangle} if it originates from a original node of degree three in $S_n$.
In Figure~\ref{fig:circle} the connecting triangles are denoted by dashed lines.

\begin{lemma}
\label{lem:triangles}
For every $d\geq 3$, $n\geq 3$ and $n'\geq 3$, every homomorphism $f:L(G_{n,d})\to L(G_{n',d})$ must map connecting triangles to connecting triangles.
\end{lemma}

\begin{proof}
Consider an arbitrary connecting triangle of $G_{n,d}$. It has the property that each of its node is adjacent to a special node. By Lemma~\ref{lem:special}, special nodes are preserved by the homomorphism $f$.
The only triangles with this property in $G_{n',d}$ are precisely the connecting triangles.
\end{proof}

\begin{proof}[Proof of Proposition \ref{prop:homo}]
By Lemma~\ref{lem:triangles}, the connecting triangles of $L(G_{n,d})$ map to the connecting
triangles of $L(G_{n',d})$. These special triangles are joined by triangles with special nodes onto a cycle of triangles. By Lemma~\ref{lem:special} these triangles with special nodes can not
map onto connecting triangles.  
Consequently, the image of two adjacent connecting triangles uniquely 
determines the image of the entire cycle: for 
the next adjacent connecting triangle in $L(G_{n,d})$, 
only one connecting triangle in $L(G_{n',d})$ is available for its image
as the adjacency of connecting triangles must be preserved. 
(See Figure~\ref{fig:circle} for an illustration.)
Therefore the homomorphism $f$ is cyclic.
\end{proof}

\begin{proof}[Proof of Theorem~\ref{thm:main}]
We apply Theorem~\ref{thm:universal} and show an embedding of $(P,\leq_P)$ 
into the homomorphism order of the interval $[K_d,K_{d+1})_{\mathcal L}$.
For the chosen $d\geq 3$ we assign every $A\in P$ a line graph $L(d,A)$ consisting of the disjoint union of graphs $L(G_{3a,d}), a\in A$.
Since any homomorphism must map connected components to connected components, 
and also as $3a|3b$ if and only if $a|b$,
we know by Proposition~\ref{prop:homo}
that $L(d,A)$ allows a homomorphism to $L(d,B)$ if and only if $A\leq_P B$.
\end{proof}

\section{Concluding remarks}
\label{sec:concluding}

Our results confirm that the homomorphism order of line graphs is rich.  It is interesting that our embedding
considerably differs from the one used in the proof of the universality of
oriented paths by Hubi\v{c}ka and Ne\v{s}et\v{r}il~\cite{Hubicka2005}.

Our construction is based on the retrospecting of a homomorphism $f:L(G)\to L(H)$
to a vertex mapping $f':V_G\to V_H$. We put $f'(v)=v'$ if
all edges incident to $v$ are mapped by $f$ to the edges 
incident to $v'$ in $H$. This mapping is not always well defined. In particular:
\begin{itemize}
\item[(a)] The image of an edge incident to a vertex of degree $1$ of $G$ is contained
in the set of edges incident to two different vertices $u$, $v$ connected by an edge in $H$. In this sense $f'$ is not a function.
\item[(b)] Edges incident to a vertex $v$ of degree $3$ in $G$ correspond to a triangle
in $L(G)$. Because the line graph of a triangle is also a triangle, the image of these edges
may thus map to a line graph of a triangle. In this case $f'(v)$ is not defined.
\end{itemize}
The basic idea behind the proof of Lemma~\ref{lem:core} is the fact that $f'$ (if
it is a function) is close to a graph homomorphism $G\to H$ with two
main differences:
\begin{itemize}
\item[(c)] It may happen that $f'(u)=f'(v)$ for two adjacent vertices of $G$.
\item[(d)] For vertices of degree at least $3$ the mapping $f'$ is locally injective with the exception of $(c)$.
\end{itemize}
A homomorphism $h:G\to H$ is {\em locally injective} if the restriction of the
mapping $h$ to the domain consisting of the vertex neighborhood of $v$ and
range consisting of the vertex neighborhood of $h(v)$ is injective.  In one
direction, every locally injective homomorphism $h:G\to H$ yields a
homomorphism $h':L(G)\to L(H)$. Our observations above show that this direction can be reversed in special cases.

It is thus not a surprise that our universality proof is based on ideas
developed for the proof of universality of locally injective homomorphisms.  We
get closer to graph homomorphisms by means of the indicator construction.  In
the proof of Proposition~\ref{prop:homo} we consider a mapping $f'':V_G\to V_H$
that retrospects a homomorphism $L(G*I_d(a,b))\to L(H*I_d(a,b))$ in a similar way
as $f'$. This mapping is a locally injective homomorphism with the exception of
$(b)$ and vertices of degree 2, where the local injecitivity is not enforced.
For this reason we use sunlets instead of cycles and our
embedding of the divisibility partial order is based on the fact that there is a
locally injective homomorphism from sunlet graphs $S_n$ to $S_m$ if and only if
$m$ divides $n$.

With this insight it is not difficult to see that our construction can be
altered to form $3$-regular graphs. This can be done by adding a cycle
consisting of all of the pendant vertices to every sunlet.  For degrees $d>3$, the edges
connecting both inner and outer can be turned into a multi-edges of a given
degree, but also the indicator needs to be modified to become $d$-regular
except for the two vertices of degree 1.  By replacing the edge connecting the
dragon with vertex $c$ by a clique of the corresponding degree and by 
adding a separate copy of a dragon to all but one vertex (the one connected to the base).

The locally injective homomorphism order is the main subject of several groups of authors~\cite{Fiala2005, Fiala}, see also a survey~\cite{Fiala2008}.
It is shown that several properties of the locally injective homomorphism order
are given by degree refinement matrices of the graphs considered. 
For a graph $G$, 
the degree refinement (also known as the equitable partition) is the coarsest 
partition of the vertex set of $G$ into classes $B_1,\dots,B_l$ s.t. vertices in the same class cannot be distinguished by counting their neighbors in these classes, i.e. for any $i$, $j$
and any $u,v\in B_i$ it holds that $|N(u)\cap B_j|=|N(v)\cap B_j|=d_{i,j}$.
When blocks $B_i$ are ordered in a canonical way, say by promoting blocks with vertices of high degree first (see \cite{Fiala, Fiala2005} for details) one 
may arrange these parameters $d_{i,j}$ into the unique {\em degree refinement matrix} $D$ of the graph $G$.
As a special case, for $d$-regular graphs $G$ the degree refinement matrix
is trivial consisting of only one value, i.e. $D=(d)$.

Roberson \cite{Roberson} shows that the homomorphism order line graphs is dense
above every complete graph $K_n$, $n\geq 2$.  Our construction gives many extra
pairs with infinitely many different graphs strictly in between. Further such pairs can be obtained
by an application of our other result~\cite{Fiala}:
\begin{thm}[\cite{Fiala}]
\label{thm:density}
Let $G$ and $H$ be connected graphs with different degree refinement matrices, there exists a locally injective homomorphism $h:G\to H$ and $H$ has no vertices of degree 1. Then:
\begin{enumerate}
\item[(a)] There exists a connected graph $F$ that is strictly in between $G$ and $H$ in the locally injective homomorphism order.
\item[(b)] When $G$ has no vertices of degree~1 and $H$ has at least one cycle with a vertex of degree greater than 2, then $F$ can be constructed to have no vertices of degree 1 and contain a cycle with a vertex of degree greater than~2.
\end{enumerate}
\end{thm}

Fix $d>2$.  Assume that $G$ and $H$ have no vertex of degree 2 and the degree of
vertices is either bounded by $d-1$, or bounded by $d$ and moreover that $G$ and $H$ are in Vizing class 1. Then graph $F$ given by Theorem~\ref{thm:density} can be easily extended to $F'$, where an extra edge is added to every vertex of degree $2$. It is also easy to see that it cannot have vertices of degree greater than $d$. From the proof of Theorem~\ref{thm:density} it follows that $F$ is in Vizing class 1.

Consequently, $L(F'*I_d(a,b))$ is strictly in between $L(G*I_d(a,b))$  and
$L(H*I_d(a,b))$. Afterwards, a graph strictly in between $G*I_d(a,b)$ and $F'*I_d(a,b)$
can be again constructed by applying Theorem~\ref{thm:density} on $G$ and $F$.
This construction provides new examples of dense pairs in the homomorphism order of
line graphs.

There are more results about the locally injective homomorphism order that seem
to suggest a strategy to attack problems about the homomorphism order of line graphs.
It appears likely that the characterization of gaps in the locally injective homomorphism
order will give new gaps in the homomorphism order of line graphs.
The proof of universality of the locally injective order of connected graphs in \cite{Fiala}
can be translated to yet another proof of the universality of homomorphism order of line graphs,
this time however the graphs used are connected but not $d$-regular --- since locally
injective homomorphism order is not universal on $d$-regular connected graphs.
This suggests the question of whether the homomorphism order of line graphs is universal on the class of finite connected $d$-regular graphs.

Finally, Leighton's construction of a common covering for graphs \cite{Leighton1982}
may give an insight into a way of constructing a suitable product for line graphs.

\section{Acknowledgment}
We would like to thank to Jaroslav Ne\v set\v ril and to the anonymous referees
for remarks that improved quality of this paper.

\bibliographystyle{plain}
\bibliography{linegraphs}

\end{document}